\documentclass[12pt]{amsart}
\usepackage{latexsym,amssymb,amsmath,lscape,young}
\usepackage{youngtab}
\usepackage{amsmath,amsthm,amssymb,amscd}
\usepackage[mathscr]{eucal}
\usepackage{verbatim}

\usepackage{color}

\newtheoremstyle{custom}
  {3pt}
  {3pt}
  {\slshape}
  {}
  {\bfseries}
  {.}
  { }
   {}
\theoremstyle{custom}
\newtheorem{theorem}{Theorem}[section]

\newtheorem{proposition}[theorem]{Proposition}
\newtheorem{proposition/definition}[theorem]{Proposition/Definition}
\newtheorem{lemma}[theorem]{Lemma}
\newtheorem{corollary}[theorem]{Corollary}

\theoremstyle{definition}
\newtheorem{definition}[theorem]{Definition}

\newtheorem{example}[theorem]{Example}

\theoremstyle{remark}
\newtheorem{remark}[theorem]{Remark}

\begin{document}
\title{The stabilizer of immanants}
\author{Ke Ye}
\date{Nov 2010}
\begin{abstract}
Immanants are homogeneous polynomials of degree $n$ in $n^{2}$ variables associated to the irreducible representations of the symmetric group $\mathfrak{S}_{n}$ of $n$ elements. We describe immanants as trivial $\mathfrak{S}_{n}$ modules and show that any homogeneous polynomial of degree $n$ on the space of $n\times n$ matrices preserved up to scalar by left and right action by diagonal matrices and conjugation by permutation matrices is a linear combination of immanants. Building on works of M. Ant{\'o}nia Duffner \cite{MR1275631} and Coelho, M. Purifica{\c{c}}{\~a}o \cite{MR1412753}, we prove that for $n\ge6$ the identity component of the stabilizer of any immanant (except determinant, permanent, and $\pi=(4,1,1,1)$) is $\Delta(\mathfrak{S}_{n})$ $\ltimes T(GL_{n}\times GL_{n})$ $\ltimes \mathbb{Z}_{2}$, where $T(GL_{n}\times GL_{n})$ is the group consisting of pairs of $n\times n$ diagonal matrices with the product of determinants $1$, acting by left and right matrix multiplication, $\Delta(\mathfrak{S}_{n})$ is the diagonal of $\mathfrak{S}_{n}\times \mathfrak{S}_{n}$, acting by conjugation, ($\mathfrak{S}_{n}$ is the group of symmetric group.) and $\mathbb{Z}_{2}$ acts by sending a matrix to its transpose. Based on the work of Coelho, M. Purifica{\c{c}}{\~a}o and Duffner, M. Ant{\'o}nia \cite{MR1674232}, we also prove that for $n\ge 5$ the stabilizer of the immanant of any non-symmetric partition (except determinant and permanent) is $\Delta(\mathfrak{S}_{n})$ $\ltimes T(GL_{n}\times GL_{n})$ $\ltimes \mathbb{Z}_{2}$.
\end{abstract}

\email{kye@math.tamu.edu}
\maketitle
\section{\textbf{Introduction}}
D.E. Littlewood \cite{MR2213154} defined polynomials of degree $n$ in $n^{2}$ variables generalizing the notion of determinant and permanent, called immanants, and are defined as follows:
\begin{definition}
For any partition $\pi\vdash n$, define a polynomial of degree n in matrix variables $(x_{ij})_{n\times n}$ associated to $\pi$ as follows:
\[
P_{\pi}:=\sum_{\sigma\in \mathfrak{S}_{n}}\chi_{\pi}(\sigma)\prod_{i=1}^{n}x_{i\sigma(i)}                 \]
This polynomial is called the immanant associated to $\pi$.
\end{definition}
\begin{example}
If $\pi=(1,1,...,1)$ then $P_{\pi}$ is exactly the determinant of the matrix $(x_{ij})_{n\times n}$.
\end{example}
\begin{example}
If $\pi=(n)$ then $P_{\pi}$ is the permanent $\sum_{\sigma\in \mathfrak{S}_{n}}$$\prod_{i=1}^{n} x_{i\sigma(i)}$.
\end{example}
\begin{remark}
Let $E$ and $F$ be $\mathbb{C}^{n}$. Since immanants are homogeneous polynomials of degree $n$ in $n^{2}$ variables, we can identify them as elements in $S^{n}(E\otimes F)$ (Identify the space $E\otimes F$ with the space of $n\times n$ matrices). The space $S^{n}(E\otimes F)$ is a representation of $GL(E\otimes F)$, in particular, it is a representation of $GL(E)\times GL(F)\subset GL(E\otimes F)$. So we can use the representation theory of $GL(E)\times GL(F)$ to study immanants. The explicit expression of an immanant $P_{\pi}$ in $S^{n}(E\otimes F)$ is:
\[
\sum_{\sigma\in
\mathfrak{S}_{n}}\chi_{\pi}(\sigma)\prod_{i=1}^{n}e_{i}\otimes f_{\sigma(i)},
\]
where $\prod$ is interpreted as the symmetric tensor product.
\end{remark}

In section $3$ we remark that immanants can be defined as trivial $\mathfrak{S}_{n}$ modules (Proposition \ref{proposition}). Duffner, M. Ant{\'o}nia found the system of equations determining the stabilizer of immanants (except determinant and permanent) for $n\ge 4$ in \cite{MR1275631} in year $1994$. $2$ years later, Coelho, M. Purifica{\c{c}}{\~a}o proved in \cite{MR1412753} that if the system of equations in \cite{MR1275631} has a solution, then permutations $\tau_{1}$ and $\tau_{2}$ in the system must be the same. Building on works of Duffner and Coelho, We prove the main results Theorem \ref{theorem1} and Theorem \ref{theorem2} of this paper in section $4$.
\begin{theorem}\label{theorem1}
Let $\pi$ be a partition of $n\ge 6$ such that $\pi\ne(1,...,1)$, $(4,1,1,1)$ or $(n)$, then the identity component of the stabilizer of the immanant $P_{\pi}$ is $\Delta(\mathfrak{S}_{n})$ $\ltimes T(GL_{n}\times GL_{n})$ $\ltimes \mathbb{Z}_{2}$, where $T(GL_{n}\times GL_{n})$ is the group consisting of pairs of $n\times n$ diagonal matrices with the product of determinants $1$, acting by left and right matrix multiplication, $\Delta(\mathfrak{S}_{n})$ is the diagonal of $\mathfrak{S}_{n}\times \mathfrak{S}_{n}$, acting by conjugation, ($\mathfrak{S}_{n}$ is the group of symmetric group.) and $\mathbb{Z}_{2}$ acts by sending a matrix to its transpose.
\end{theorem}
\begin{remark}
It is well-known that Theorem \ref{theorem1} is true for permanent as well, but our proof does not recover this case.
\end{remark}
\begin{theorem}\label{theorem2}
Let $n\ge 5$ and let $\pi$ be a partition of $n$ which is not symmetric, that is, $\pi$ is not equal to its transpose, and $\pi\ne (1,...,1)$ or $(n)$, then the stabilizer of the immanant is $\Delta(\mathfrak{S}_{n})\ltimes T(GL_{n}\times GL_{n})\ltimes \mathbb{Z}_{2}$.
\end{theorem}
\begin{remark}One can compute directly from the system of equations determined by Duffner, M. Ant{\'o}nia for the case $n=4$ and $\pi=(2,2)$ and see that in this case, Theorem \ref{theorem2} fails, since there will be many additional components. For example,
\[
C=\begin{pmatrix}
  e&-e&-e&e\\
  1&1&1&1\\
  1&1&-1&-1\\
  \frac{1}{e}&-\frac{1}{e}&\frac{1}{e}&-\frac{1}{e}
\end{pmatrix}
\] stabilizes the immanant $P_{(2,2)}$, but it's not in the identity component.
\end{remark}
\section{\textbf{Notations and preliminaries}}
\begin{enumerate}
\item $E$ and $F$ are $n$-dimensional complex vector spaces.
\item $\lbrace e_{i}\rbrace_{i=1}^{n}$ and $
    \lbrace f_{i}\rbrace_{i=1}^{n}$ are fixed basis of $E$ and $F$, respectively.
\item $\mathfrak{S}_{n}$ is the symmetric group on $n$ elements. Given $\sigma\in \mathfrak{S}_{n}$, we can express $\sigma$ as disjoint product of cycles, and we can denote the conjugacy class of $\sigma$ by $(1^{i_{1}}2^{i_{2}}...n^{i_{n}})$, meaning that $\sigma$ is a disjoint product of $i_{1}$ $1$-cycles, $i_{2}$ $2$-cycles,..., $i_{n}$ $n$-cycles. Sometimes we might use $(k_{1},k_{2},...,k_{p})$ (where $n\ge k_{1}\ge k_{2}\ge ...\ge k_{p}\ge 1$ and $\sum_{i=1}^{p}k_{i}=n$) to indicate the cycle type of $\sigma$. This notation means that $\sigma$ contains a $k_{1}$-cycle, a $k_{2}$-cycle $\dots$ and a $k_{p}$-cycle.
\item $\pi \vdash{n}$ is a partition of $n$, we can write $\pi=(\pi_{1},...,\pi_{k})$ with $\pi_{1}\ge...\ge \pi_{k}$.
\item $\pi^{'}$ is the conjugate partition of $\pi$.
\item $M_{\pi}$ (sometimes we might use $[\pi]$ as well) is the irreducible representation of the symmetric group $\mathfrak{S}_{n}$ corresponding to the partition $\pi$.
\item $S_{\pi}(E)$ is the irreducible representation of $GL(E)$ corresponding to $\pi$.
\item $\chi_{\pi}$ is the character of $M_{\pi}$.
\item Let $V$ be a representation of $SL(E)$, then the weight-zero-subspace of $V$ is denoted as $V_{0}$.
\item For $S_{\pi}(E)$, we have a realization in the tensor product $E^{\otimes n}$ via the young symmetrizer $c_{\overline \pi}$\;:\;$S_{\pi}(E)\cong c_{\overline \pi}(E^{\otimes n})$,where $\overline \pi$ is any young tableau of shape $\pi$.
\item The action of $\mathfrak{S}_{n}$ on $E$ is the action of the Weyl group of $SL(E)$, that is, the permutation representation on $\mathbb{C}^{n}$.
\item Fix a young tableau $\overline \pi$, define $P_{\bar \pi}:=\{\sigma\in \mathfrak{S}_{n}\mid \sigma \text{ preserves each row of}\; \bar \pi\}$
and $Q_{\bar \pi}:=\{\sigma\in \mathfrak{S}_{n}\mid \sigma \text{ preserves each column of}\; \bar \pi\}$
\item For a polynomial $P$ in variables $(x_{ij})_{n\times n}$, denote the stabilizer of $P$ in $GL(E\otimes F)$ by $G(P)$.
\end{enumerate}
\section{\textbf{The description of immanants as modules}}
Consider the action of $T(E)\times T(F)$ on immanants, where $T(E)$,$T(F)$ are maximal tori (diagonal matrices) of $SL(E)$, $SL(F)$, respectively. For any $(A,B)\in T(E)\times T(F)$,
\[\mathbf{A}=
\begin{pmatrix}a_{1}&0&\dots&0\\
      0&a_{2}&\dots&0\\
      \vdots&\vdots&\ddots&\vdots\\
      0&0&\dots&a_{n}
      \end{pmatrix},
\mathbf{B}=\begin{pmatrix}b_{1}&0&\dots&0\\
      0&b_{2}&\dots&0\\
      \vdots&\vdots&\ddots&\vdots\\
      0&0&\dots&b_{n}
      \end{pmatrix}
\]
For the immanant $P_{\pi}=\sum_{\sigma\in \mathfrak{S}_{n}}\chi_{\pi}(\sigma)\prod_{i=1}^{n}x_{i\sigma(i)}$, the action of $(A,B)$ on $P_{\pi}$ is given by
\[(A,B).P_{\pi}:=\sum_{\sigma\in \mathfrak{S}_{n}}\chi_{\pi}(\sigma)\prod_{i=1}^{n}a_{i}b_{\sigma(i)}x_{i\sigma(i)}=P_{\pi}.
\]
That is, immanants are in the $SL(E)\times SL(F)$ weight-zero-subspace of $S^{n}(E\otimes F)$. On the other hand, the decomposition of $S^{n}(E\otimes F)$ as $GL(E)\times GL(F)$-modules:
\[
S^{n}(E\otimes F)=\sum_{\lambda \vdash n}S_{\lambda}(E)\otimes S_{\lambda}(F).
\]
Accordingly, we have a decomposition of the weight-zero-subspace:
\[
(S^{n}(E\otimes F))_{0}=\sum_{\lambda \vdash n}(S_{\lambda}(E))_{0}\otimes (S_{\lambda}(F))_{0}.
\]
\begin{proposition}
For $\lambda\vdash n$, $(S_{\lambda}(E))_{0}\cong M_{\lambda}$ as $\mathfrak{S}_{n}$-modules.
\end{proposition}
\begin{proof}
See \cite{MR2265844} page 272.
\end{proof}
Thus we can identify $(S_{\lambda}(E))_{0}\otimes (S_{\lambda}(F))_{0}$ with $M_{\lambda}\otimes M_{\lambda}$ as an $\mathfrak{S}_{n}\times \mathfrak{S}_{n}$ module. Also, the diagonal $\Delta(\mathfrak{S}_{n})$ of $\mathfrak{S}_{n}\times \mathfrak{S}_{n}$ is isomorphic to $\mathfrak{S}_{n}$, so $M_{\lambda}\otimes M_{\lambda}$ is a $\mathfrak{S}^{n}$-module. $M_{\lambda}\otimes M_{\lambda}$ is an irreducible $\mathfrak{S}_{n}\times \mathfrak{S}_{n}$ module, but it is reducible as and $\mathfrak{S}_{n}$-module, so that we can decompose it.Consider the action of $\mathfrak{S}^{n}$ on $S^{n}(E\otimes F)$, let $\sigma\in \mathfrak{S}^{n}$, then:
\[
\sigma.x_{ij}=x_{\sigma(i)\sigma(j)}.
\]
So immanants are invariant under the action of $\mathfrak{S}^{n}$, hence are contained in the isotypic component of the trivial $\mathfrak{S}_{n}$ representation of $\bigoplus_{\lambda \vdash n}(S_{\lambda}(E))_{0}\otimes (S_{\lambda}(F))_{0}$=$\bigoplus_{\lambda \vdash n}M_{\lambda}\otimes M_{\lambda}$.
\begin{proposition}
As a $\mathfrak{S}_{n}$ module, $M_{\lambda}\otimes M_{\lambda}$ contains only one copy of trivial representation.
\end{proposition}
\begin{proof}
Denote the character of $\sigma\in \mathfrak{S}_{n}$ on $M_{\lambda}\otimes M_{\lambda}$ by $\chi(\sigma)$, and let $\chi_{trivial}$ be the character of the trivial representation. From the general theory of characters, it suffices to show that the inner product $(\chi,\chi_{trivial})=1$. First, the character $\chi(\sigma,\tau)$ of $(\sigma,\tau) \in \mathfrak{S}_{n}\times \mathfrak{S}_{n}$ on the module $M_{\lambda}\otimes M_{\lambda}$ is $\chi_{\lambda}(\sigma)\chi_{\lambda}(\tau)$. So in particular, the character $\chi$ of $M_{\lambda}\otimes M_{\lambda}$ on $\sigma$ is $(\chi_{\lambda}(\sigma))^{2}$. Next,
\[
\begin{aligned}
(\chi,\chi_{trivial})=&\frac{1}{n!}(\sum_{\sigma\in \mathfrak{S}_{n}}\chi(\sigma)\overline{\chi_{trivial}(\sigma)})\\
=&\frac{1}{n!}\sum_{\sigma\in \mathfrak{S}_{n}}(\chi_{\lambda}(\sigma))^{2}\\
=&(\chi_{\lambda},\chi_{\lambda})\\
=&1
\end{aligned}
\]
since $\chi_{trivial}(\sigma)=1,\forall\; \sigma\in \mathfrak{S}_{n}$. \end{proof}
By the above proposition, $M_{\lambda}\otimes M_{\lambda}=\mathbb{C}_{\lambda}\oplus \cdots$, where $\mathbb{C}_{\lambda}$ means the unique copy of trivial representation in $M_{\lambda}\otimes M_{\lambda}$, and dots means other components in this module. Hence
\[
P_{\pi}\in\bigoplus_{\lambda \vdash n}\mathbb{C}_{\lambda}.
\]
We can further locate immanants:
\begin{proposition}\label{proposition}
Let $P_{\pi}$ be the immanant associated to the partition $\pi\vdash n$. Assume $\mathbb{C}_{\pi}$ is the unique copy of the trivial $\mathfrak{S}_{n}$-representation contained in $(S_{\pi}(E))_{0}\otimes (S_{\pi}(F))_{0}$. Then:
$P_{\pi}\in\mathbb{C}_{\pi}.$
\end{proposition}
Before proving this proposition, we remark that it gives an equivalent definition of the immanant: $P_{\pi}$ is the element of the trivial representation $\mathbb{C}_{\pi}$ of $M_{\pi}\otimes M_{\pi}$ such that $P_{\pi}(Id)=dim([\pi])$. For more information about this definition, see for example, \cite{MR1771845}.
\begin{example}
If $\pi=(1,...,1)$, then $S_{\pi}E=\bigwedge^{n}E$, which is already a $1$ dimensional vector space. If $\pi=(n)$, then $S_{\pi}E=S^{n}E$, in which there's only one (up to scale) weight zero vector $e_{1}\circ\cdots\circ e_{n}$.
\end{example}
\begin{proof}Fix a partition $\pi\vdash n$, we want to show that the immanant $P_{\pi}$ is in $\mathbb{C}_{\pi}$, but we know that $P_{\pi}$ is in the weight-zero-subspace $(S^{n}(E\otimes F))_{0}=\oplus_{\lambda\vdash n}(S_{\lambda}(E))_{0}\otimes(S_{\lambda}(F))_{0}$.Since
$(S_{\lambda}(E))_{0}\otimes(S_{\lambda}(F))_{0}\subset S_{\lambda}(E)\otimes S_{\lambda}(F)$, it suffices to show that $P_{\pi}\in S_{\pi}(E)\otimes S_{\pi}(F)$. Then it suffices to show that for any young symmetrizer $c_{\bar\lambda}$ not of the shape $\pi$, $c_{\bar\lambda}\otimes c_{\bar\lambda}(P_{\pi})=0$. It suffices to check that $1\otimes c_{\bar\lambda}(P_{\pi})=0$, since $c_{\bar{\lambda}}\otimes c_{\bar\lambda}=(c_{\bar{\lambda}}\otimes 1)\circ (1\otimes c_{\bar\lambda})$. Express $P_{\pi}$ as an element in $S^{n}(E\otimes F)$:
\[
P_{\pi}=\sum_{\sigma\in \mathfrak{S}_{n}}\sum_{\tau\in \mathfrak{S}_{n}}\chi_{\pi}(\sigma)(\otimes_{i=1}^{n}e_{\tau(i)})\otimes(\otimes_{i=1}^{n}f_{\sigma\circ\tau(i)})
\]
The young symmetrizer $c_{\bar\lambda}=\sum_{p\in P_{\lambda},\;q\in Q_{\lambda}}sgn(q)\;p\;q$. So
\[
\begin{aligned}
1\otimes c_{\lambda}(P_{\pi})=&\sum_{\tau\in \mathfrak{S}_{n}}\otimes_{i=1}^{n}e_{\tau(i)}\otimes (\sum_{\substack {p\in P_{\lambda}
q\in Q_{\lambda}\\
\sigma\in \mathfrak{S}_{n}}}\chi_{\pi}(\sigma)sgn(q)\otimes_{i=1}^{n}f_{\sigma\cdot\tau\cdot q\cdot p(i)})
\end{aligned}
\]
In the above expression, $c_{\bar\lambda}$ acts on $f_{i}$'s. Now it suffices to show that:
\[
\sum_{\substack {p\in P_{\lambda}\\q\in Q_{\lambda}\\\sigma\in \mathfrak{S}_{n}}}\chi_{\pi}(\sigma)sgn(q)\bigotimes_{i=1}^{n}f_{\sigma\cdot\tau\cdot q\cdot p(i)}=0,\forall\; \tau\in \mathfrak{S}_{n}.
\]
For any $\tau\in \mathfrak{S}_{n}$,
\[
\begin{aligned}
\sum_{\substack {p\in P_{\lambda}\\q\in Q_{\lambda}\\\alpha\in \mathfrak{S}_{n}}}\chi_{\pi}(\alpha\cdot\tau^{-1})sgn(q)\bigotimes_{i=1}^{n}f_{\alpha\cdot q\cdot p(i)}=&\sum_{\gamma\in \mathfrak{S}_{n}}
\sum_{\substack {p\in P_{\lambda}\\q\in Q_{\lambda}\\
\alpha\in \mathfrak{S}_{n}\\
\alpha\cdot q\cdot p=\gamma
}}\chi_{\pi}(\alpha\cdot\tau^{-1})sgn(q)\bigotimes_{i=1}^{n}f_{\gamma(i)}\\
=&\sum_{\gamma\in \mathfrak{S}_{n}}[\sum_{\substack {p\in P_{\lambda}\\q\in Q_{\lambda}}}\chi_{\pi}(\gamma\cdot p^{-1}\cdot q^{-1}\cdot\tau^{-1})sgn(q)]\bigotimes_{i=1}^{n}f_{\gamma(i)}.
\end{aligned}
\]
Let $\sigma$,$\tau$ $\in \mathfrak{S}_{n}$, $\sigma\cdot\tau=\sigma\cdot\tau\cdot\sigma^{-1}\cdot\sigma$, so $\sigma\cdot\tau=\tau^{'}\cdot\sigma$, where $\tau^{'}$ is conjugate to $\tau$ in $\mathfrak{S}_{n}$ by $\sigma$. Therefore, we can rewrite the previous equation as:
\[
\sum_{\gamma\in \mathfrak{S}_{n}}[\sum_{\substack {p\in P_{\lambda}\\q\in Q_{\lambda}}}\chi_{\pi}(\tau^{-1}\cdot\gamma\cdot p^{-1}\cdot q^{-1})sgn(q)]\bigotimes_{i=1}^{n}f_{\gamma(i)}.
\]\\
Therefore, it suffices to show:
\[
\sum_{\substack {p\in P_{\bar\lambda}\\q\in Q_{\bar\lambda}}}\chi_{\pi}(\gamma\cdot p\cdot q)sgn(q)=0,\forall\; \gamma\in \mathfrak{S}_{n}
\]
This equality holds because the left hand side is the trace of $\gamma\cdot c_{\bar\lambda}$ as an operator on the space $\mathbb{C}\mathfrak{S}_{n}\cdot c_{\bar\pi}$,the group algebra of $\mathfrak{S}_{n}$, which is a realization of $M_{\pi}$ in $\mathbb{C}\mathfrak{S}_{n}$, but this operator is in fact zero:
\[\forall\text{ } \sigma\in \mathfrak{S}_{n}\text{ , }
\gamma\cdot c_{\bar\lambda}\cdot\sigma\cdot c_{\bar\pi}=\gamma\cdot\sigma\cdot c_{\bar\lambda^{'}}\cdot c_{\bar\pi}=0,
\]
where $\bar\lambda^{'}$ is the young tableau of shape $\lambda^{'}$ which is conjugate to $\bar\lambda$ by $\sigma$. This implies that $c_{\bar\lambda}$ and $c_{\bar\pi}$ are of different type, hence $c_{\bar\lambda}\cdot c_{\bar\pi}=0$, in particular, the trace of this operator is $0$. Therefore, $P_{\pi}\in S_{\pi}(E)\otimes S_{\pi}(F)$.
\end{proof}
\begin{corollary}
If a homogeneous polynomial $Q$ of degree $n$ is preserved by the group $\Delta(\mathfrak{S}_{n})\times T(E)\times T(F)$, then $Q$ is a linear combination of immanants. Furthermore, immanants are linearly independent and form a basis of the space of all homogeneous degree $n$ polynomials preserved by $\Delta(\mathfrak{S}_{n})\times T(E)\times T(F))$.
\end{corollary}
\begin{proof} If $Q$ is preserved by $\mathfrak{S}_{n}\times T(GL(E)\times GL(F))$, then $Q$ is in $\bigoplus _{\lambda\vdash n}\mathbb{C}_{\lambda}$. By the proposition, immanants form a basis of  $\bigoplus _{\lambda \vdash n}\mathbb{C}_{\lambda}$, the corollary follows.
\end{proof}
\section{\textbf{The stabilizer of immanant}}
Next, we study the stabilizer of immanants in the group $GL(E\otimes F)$.
\begin{example}:
For $\pi=(1,1,...,1)$ and $\pi=(n)$, $G(P_{\pi})$ are well-known: If $\pi=(1,1,...,1)$,then $G(P_{\pi})=S(GL(E)\times GL(F))\ltimes \mathbb{Z}_{2}$, and if $\pi=(n)$,then $G(P_{\pi})=\mathfrak{S}_{n}\times \mathfrak{S}_{n}\times T(GL(E)\times GL(F))\ltimes \mathbb{Z}_{2}$, where $S(GL(E)\times GL(F)$) is a subgroup of $GL(E)\times GL(F)$ consisting of pairs $(A,B)$ with $det(A)det(B)=1$, $T(GL(E)\times GL(F))$ is the pair of diagonal matrices with the product of determinants 1, and $\ltimes\mathbb{Z}_{2}$ means that we are allowed to take the transpose of matrices. For the stabilizer of determinant, see G.Frobenius \cite{MR2213154}. For the stabilizer of permanent, see Botta \cite{MR0213376}.
\end{example}
Assume $C=(c_{ij})$ and $X=(x_{ij})$ are $n\times n$ matrices. Denote the torus action of $C$ on $X$ by $C\ast X=Y$, where $Y=(y_{ij})$ is an $n\times n$ matrix with entry $y_{ij}=c_{ij}x_{ij}$. Note that the torus action is just the action of the diagonal matrices in $End(E\otimes F)$ on the vector space $E\otimes F$. To find $G(P_{\pi})$, we need the following result from \cite{MR1275631}:
\begin{theorem}[\cite{MR1275631}]\label{Duffner}
Assume  $n\ge4$, $\pi\ne (1,1,...,1)$ and $(n)$. A linear transformation $T\in GL(E\otimes F)$ preserves the immanant $P_{\pi}$ iff $T\in T(GL(E\otimes F))\ltimes \mathfrak{S}_{n}\ltimes \mathfrak{S}_{n}\ltimes \mathbb{Z}_{2}$, and satisfies the relation:
\[
\chi_{\pi}(\sigma)\prod_{i=1}^{n}c_{i\sigma(i)}=\chi_{\pi}(\tau_{2}\sigma\tau_{1}^{-1}),
\]
where  $\sigma$ runs over all elements in $\mathfrak{S}_{n}$, $T(GL(E\otimes F))$ is the torus of $GL(E\otimes F)$, acting by the torus action described above, $\mathfrak{S}_{n}$ is the symmetric group in $n$ elements, acting by left and right multiplication, and $\mathbb{Z}_{2}$ sending a matrix to its transpose.
\end{theorem}
\begin{proof}[sketch]

Step 1: Let $\pi$ be a fixed partition of $n$. Define a subset $\mathfrak{A}$ of the set $M_{n}(\mathbb{C})$ of $n$ by $n$ matrices as follows, $X^{(n-1)}:=\lbrace A\in M_{n}(\mathbb{C}):\text{degree of } P_{\pi}(xA+B)\leqslant1,\text{for every } B\in M_{n}(\mathbb{C})\rbrace$. Geometrically, $X^{(n-1)}$ is the most singular locus of $IM_{\pi}=0$. If $A$ is in $X^{(n-1)}$, and $T$ preserves $P_{\pi}$ (it can be shown that $T$ is invertible), then we have that $T(A)\in X^{(n-1)}$, since the preserver of the hypersurface $IM_{\pi}=0$ will preserve the most singular locus as well.\\
Step 2: Characterize the set $X^{(n-1)}$. To do this, first define a subset $R_{i}$ (resp. $R^{i}$) of $M_{n}(\mathbb{C})$, consisting of matrices that have nonzero entries only in $i$-th row (resp. column). Then one proves that $A\in X^{(n-1)}$ if and only if it is in one of the forms:\\
\begin{enumerate}
\item $R_{i}$ or $R^{i}$ for some $i$.
\item The nonzero elements are in the $2\times 2$ submatrix $A[i,h\mid i,h]$, and
\[
\chi_{\pi}(\sigma)a_{ii}a_{hh}+\chi_{\pi}(\tau)a_{ih}a_{hi}=0
\]for every $\sigma$ and $\tau$ satisfying $\sigma(i)=i$, $\tau(h)=h$ and $\tau=\sigma(ih)$.\\
\item $\pi=(2,1,...,1)$ and there are complementary sets of indices $\lbrace i_{1},...,i_{p} \rbrace$, $\lbrace j_{1},...,j_{q} \rbrace$ such that the nonzero elements are in $A[i_{1},...,i_{p}\mid j_{1},...,j_{q}]$ and the rank of $A[i_{1},...,i_{p}\mid j_{1},...,j_{q}]$ is one.\\
\item $\pi=(n-1,1)$, the nonzero elements are in a $2$ by $2$ submatrix $A[u,v\mid r,s]$, and the permanent of this submatrix is zero.\\
\end{enumerate}
Step 3: Characterize $T$ by sets $R_{i}$ and $R^{j}$.
\end{proof}

We will start from this theorem. From this theorem, we know that $G(P_{\pi})$ is contained in the group $\mathfrak{S}_{n}\times \mathfrak{S}_{n}\times \mathbb{C}^{n^{2}}\ltimes\mathbb{Z}^{2}$, and subject to the relation in Theorem \ref{Duffner}.
\begin{remark}
In the equation in the Theorem \ref{Duffner}, $c_{ij}\ne0, \forall\; 1\leqslant i,j\leqslant n$, since the stabilizer of $P_{\pi}$ is a group.
\end{remark}
Now instead of considering $n^{2}$ parameters, we can consider $n!$ parameters, ${i.e}$, consider the stabilizer of immanant in the bigger monoid $\mathfrak{S}_{n}\times \mathfrak{S}_{n}\times \mathbb{C}^{n!}\ltimes \mathbb{Z}^{2}$. We can ignore the $\mathbb{Z}^{2}$-part of this monoid. The action of the monoid on the weight zero space of $S^{n}(E\otimes F)$ spanned by monomials  $x_{1\sigma(1)}x_{2\sigma(2)}\cdots x_{n\sigma(n)}, \sigma\in\mathfrak{S}_{n}$ is:
\[
(\tau_{1},\tau_{2},(c_{\sigma})_{\sigma\in \mathfrak{S}_{n}})\cdot (x_{1\sigma(1)}x_{2\sigma(2)}\cdots x_{n\sigma(n)})=c_{\sigma}
x_{\tau_{1}(1)\tau_{2}\sigma(1)}x_{\tau_{1}(2)\tau{2}\sigma(2)}\cdots x_{\tau_{1}(n)\tau_{2}\sigma(n)}.
\]
\begin{proposition}
The stabilizer of $P_{\pi}$ in $\mathfrak{S}_{n}\times \mathfrak{S}_{n}\times \mathbb{C}^{n!}$ is determined by equations
\begin{equation}\label{E2}
c_{\tau_{2}^{-1}\tau_{1}\sigma}\chi_{\pi}(\tau_{2}^{-1}\tau_{1}\sigma)=\chi_{\pi}(\sigma),\forall\;\sigma\in\mathfrak{S}_{n}
\end{equation}
\end{proposition}
\begin{proof}
The action of this monoid on $P_{\pi}$ is:
\[
\begin{aligned}
(\tau_{1},\tau_{2},(c_{\sigma})_{\sigma\in \mathfrak{S}_{n}})\cdot P_{\pi}=&\sum_{\sigma\in\mathfrak{S}_{n}}\chi_{\pi}(\sigma)c_{\sigma}\prod_{i=1}^{n}x_{\tau_{1}(i),\tau_{2}\sigma(i)}\\
=&\sum_{\sigma\in\mathfrak{S}_{n}}\chi_{\pi}(\sigma)c_{\sigma}\prod_{i=1}^{n}x_{i,\tau_{2}\sigma\tau_{1}^{-1}(i)}\\
=&\sum_{\sigma\in\mathfrak{S}_{n}}\chi_{\pi}(\tau_{2}^{-1}\sigma\tau_{1})c_{\tau_{2}^{-1}\sigma\tau_{1}}\prod_{i=1}^{n}x_{i,\sigma(i)}\\
=&\sum_{\sigma\in\mathfrak{S}_{n}}\chi_{\pi}(\tau_{2}^{-1}\tau_{1}(\tau_{1}^{-1}\sigma\tau_{1}))c_{\tau_{2}^{-1}\tau_{1}(\tau_{1}^{-1}\sigma\tau_{1})}\prod_{i=1}^{n}x_{i,\sigma(i)}
\end{aligned}
\]
If $(\tau_{1},\tau_{2},(c_{\sigma}))$ stabilizes $P_{\pi}$, then
\[
\begin{aligned}
\chi_{\pi}(\tau_{2}^{-1}\tau_{1}(\tau_{1}^{-1}\sigma\tau_{1}))c_{\tau_{2}^{-1}\tau_{1}(\tau_{1}^{-1}\sigma\tau_{1})}
=&\chi_{\pi}(\sigma)=\chi_{\pi}(\tau_{1}^{-1}\sigma\tau_{1}),\forall\; \sigma\in\mathfrak{S}_{n}.
\end{aligned}
\]
Therefore, we have: $c_{\tau_{2}^{-1}\tau_{1}\sigma}\chi_{\pi}(\tau_{2}^{-1}\tau_{1}\sigma)=\chi_{\pi}(\sigma),\forall\; \sigma\in \mathfrak{S}_{n}.$
\end{proof}
Our next task is to find $\tau_{1}$ and $\tau_{2}$, such that the equation (\ref{E2}) has solution for $(c_{\sigma})_{\sigma\in \mathfrak{S}_{n}}$. For convenience, in the equation ($\ref{E2}$), set $\tau_{2}^{-1}\tau_{1}=\tau$, so we get a new equation:
\begin{equation}\label{E3}
c_{\tau\sigma}\chi_{\pi}(\tau\sigma)=\chi_{\pi}(\sigma),\forall\;\sigma\in\mathbb{S}_{n}
\end{equation}
\begin{lemma}\label{lemma2}
If the equation (\ref{E3}) has a solution then $\tau\in\mathfrak{S}_{n}$ satisfies:
\begin{enumerate}
\item If $\chi_{\pi}(\sigma)=0$, then $\chi_{\pi}(\tau\sigma)=0$;
\item If $\chi_{\pi}(\sigma)\ne0$, then $\chi_{\pi}(\tau\sigma)\ne0$;
\end{enumerate}
\end{lemma}
\begin{proof}
Clear.
\end{proof}
\begin{definition}:
For a fixed partition $\pi\vdash n$, define:
\[P:=\lbrace\sigma\in\mathfrak{S}_{n}\mid\chi_{\pi}(\sigma)=0\rbrace,\]
\[Q:=\lbrace\sigma\in\mathfrak{S}_{n}\mid\chi_{\pi}(\sigma)\ne0\rbrace.\]
\[G:=(\cap_{\sigma\in P}P\sigma)\cap(\cap_{\sigma\in Q}Q \sigma).\]
\end{definition}
\begin{lemma}
The equation ($\ref{E3}$) has a solution, then $\tau\in G$.
\end{lemma}
\begin{proof}
It suffices to show that the $2$ conditions in Lemma $\ref{lemma2}$ imply $\tau\in G$. If $\tau$ satisfies conditions $1$ and $2$, then
\[
\tau\sigma\in P,\;\forall\sigma\in P;\tau\sigma^{'}\in Q,\;\forall\sigma^{'}\in Q,
\]
therefore
\[
\tau\in P\sigma^{-1},\;\forall\;\sigma\in P;\tau\in Q\sigma^{'-1},\;\forall\;\sigma^{'}\in Q,
\]
so
\[
\tau\in G.
\]
\end{proof}
\begin{example}
We can compute $G$ directly for small $n$. If $n=3$, then we have three representations $M_{(3)}$, $M_{(1,1,1)}$, $M_{(2,1)}$, $G=\mathfrak{S}_{3}, A_{3}, A_{3}$, respectively. If $n=4$, then we have five representations $M_{(4)}$, $M_{(1,1,1,1)}$, $M_{3,1}$, $M_{(2,1,1)}$, $M_{(2,2)}$, Here $G=\mathfrak{S}_{4}$, $A_{4}$, $\lbrace (1), (12)(34), (13)(24), (14)(23)\rbrace$, $\lbrace (1), (12)(34), (13)(24), (14)(23)\rbrace$, $A_{4}$, respectively. If the partition $\pi=(1,1,...,1)$, then $G=A_{n}$. If the partition $\pi=(n)$, then $G=\mathfrak{S}_{n}$. Note that in these examples, $G$ is a normal subgroup. In fact, this holds in general.
\end{example}
The following proposition is due to Coelho, M. Purifica{\c{c}}{\~a}o in \cite{MR1412753}, using Murnagham-Nakayama Rule. One can give a different proof using Frobenius character formula for cycles (see, for example, \cite{MR1153249}).
\begin{proposition}[\cite{MR1412753}]\label{Coelho}
For any $n\ge5$ and partition $\pi\vdash n$, $G$ is a normal subgroup of $\mathfrak{S}_{n}$. Moreover, if $\pi\ne(1,1,...,1)$ or $(n)$, then $G$ is the trivial subgroup $(1)$ of $\mathfrak{S}_{n}$.
\end{proposition}
\begin{proof}[Sketch]
It is easy to show that $G$ is a normal subgroup of $\mathfrak{S}_{n}$. And then one can prove $G\ne \mathfrak{S}_{n}$ by computing character $\chi_{\pi}$. Then assume $G=A_{n}$, one can show that $Q=A_{n}$ and $P=\mathfrak{S}_{n}-A_{n}$. If such a partition $\pi$ exists, then it must be symmetric. But then one can construct cycles $\sigma$ contained in $A_{n}$ case by case (using Murnagham-Nakayama rule or Frobenius's character formula) such that $\chi_{\pi}(\sigma)=0$. It contradicts that $Q=A_{n}$.
\end{proof}
Let's return to the equation (see Theorem \ref{Duffner}):
\[
\chi_{\pi}(\sigma)\prod_{i=1}^{n}c_{i\sigma(i)}=\chi_{\pi}(\tau_{2}\sigma\tau_{1}^{-1}) \hspace{15mm}\forall\;\sigma\in \mathfrak{S}_{n}
\]
By Proposition \ref{Coelho}, we can set $\tau_{1}=\tau_{2}$ in the above equation then we have equations for $c_{\sigma}$'s:
\begin{equation}\label{E4}
\prod_{i=1}^{n}c_{i\sigma(i)}=1\; \forall\; \sigma\in\mathfrak{S}_{n}, \;\text{with}\; \chi_{\pi}(\sigma)\ne0.
\end{equation}
So elements in $G(P_{\pi})$ can be expressed as triples $(\tau,\tau,(c_{ij}))$ where matrices $(c_{ij})$ is determined by equation(\ref{E4}).
\begin{remark}
The coefficients of those linear equations are $n\times n$ permutation matrices. If we ignore the restriction $\chi_{\pi}(\sigma)\ne0$,then we get all $n\times n$ permutation matrices.
\end{remark}
\begin{lemma}\label{maxrank}
The permutation matrices span a linear space of dimension $(n-1)^2+1$ in $Mat_{n\times n}\cong \mathbb{C}^{n^{2}}$.
\end{lemma}
\begin{proof}
Consider the action of $\mathfrak{S}_{n}$ on $\mathbb{C}^{n}$ by permuting the entries, so $\sigma\in\mathfrak{S}_{n}$ is an element in $End(\mathbb{C}^{n})$, corresponding to a permutation matrix, and vice versa. Now as $\mathfrak{S}_{n}$ modules, $\mathbb{C}^{n}\cong M_{(n-1,1)}\oplus \mathbb{C}$, where $\mathbb{C}$ is the trivial representation of $\mathfrak{S}_{n}$. So we have a decomposition of vector spaces:
\[
\begin{aligned}
End(\mathbb{C}^{n})\cong& End(M_{(n-1,1)}\oplus \mathbb{C})\\
\cong& End(M_{(n-1,1)})\oplus End(\mathbb{C})\oplus Hom(M_{(n-1,1)},\mathbb{C})\oplus Hom(\mathbb{C},M_{(n-1,1)}).
\end{aligned}
\]
Since $\mathbb{C}$ and $M_{(n-1,1)}$ are $\mathfrak{S}_{n}$ modules, $\mathfrak{S}_{n}\hookrightarrow End(M_{(n-1,1)})\oplus End(\mathbb{C})$. Note that $dim(End(M_{(n-1,1)})\oplus End(\mathbb{C}))=(n-1)^{2}+1$, so it suffices to show that $\mathfrak{S}_{n}$ will span $End(M_{(n-1,1)})\oplus End(\mathbb{C})$, but this is not hard to see because we have an algebra isomorphism:
\[\mathbb{C}[\mathfrak{S}_{n}]\cong\bigoplus_{\lambda\vdash n}End(M_{\lambda})
\]
Hence, we have:
\[
\mathbb{C}[\mathfrak{S}_{n}] \twoheadrightarrow End(M_{(n-1,1)})\oplus End(\mathbb{C}).
\]
\end{proof}
\begin{remark}Lemma \ref{maxrank} shows that the dimension of the stabilizer of an immanant is at least $2n-2$. We will show next that for any partition $\pi$ of $n\ge 5$ except $(1,...,1)$ and $(n)$, the dimension of the stabilizer $G$ ($P_{\pi}$) is exactly $2n-2$.
\end{remark}
We compute the Lie algebra of the stabilizer of $G(P_{\pi})$. Since $G(P_{\pi})$ $\subset$ $GL(E\otimes F)$, the Lie algebra of $G(P_{\pi})$ is a subalgebra of $\mathfrak{g}\mathfrak{l}(E\otimes F)$. We have the decompsition of $\mathfrak{gl}$ ($E\otimes F$)
\[
\begin{aligned}
\mathfrak{gl}(E\otimes F)
&= End(E\otimes F)\\
&\cong (E\otimes F)^{*} \otimes (E\otimes F)\\
&\cong E^{*} \otimes E \otimes F^{*}\otimes F\\
&\cong (\mathfrak{s}\mathfrak{l}_{R}(E) \oplus Id_{E}\oplus T(E)) \otimes (\mathfrak{s}\mathfrak{l}_{R}(F) \oplus Id_{F}\oplus T(F))\\
&\cong \mathfrak{s}\mathfrak{l}_{R}(E)\otimes \mathfrak{s}\mathfrak{l}_{R}(F)\oplus \mathfrak{s}\mathfrak{l}_{R}(E)\otimes Id_{F}\oplus Id_{E}\otimes \mathfrak{s}\mathfrak{l}_{R}(F)\oplus Id_{E}\otimes Id_{F}\\
&\oplus T(E)\otimes T(F)\oplus T(E)\otimes Id_{F}\oplus Id_{E}\otimes T(F)
\end{aligned}
\]Where $\mathfrak{s}\mathfrak{l}_{R}(E)$ is the root space of $\mathfrak{s}\mathfrak{l}(E)$, $T(E)$ is the torus of $\mathfrak{s}\mathfrak{l}(E)$, and $Id_{E}$ is the space spanned by identity matrix. The similar notation is for $F$. We will show that the Lie algebra of $\{C\in M_{n\times n}\mid P_{\pi}(C*X)=P_{\pi}(X)\}$ is $T(E)\otimes Id_{F}\oplus Id_{E}\otimes T(F)$. Let $\{e_{i}$ $\mid$ $i=1,...,n\}$ be a fixed basis of $E$ and $\{\alpha^{i}$ $\mid$ $i=1,...,n\}$ be the dual basis. Let $H^{E}_{1i}=\alpha^{1}\otimes e_{1}-\alpha^{i}\otimes e_{i}$. Then $\{ H_{1i}$ $\mid i=2,...,n\}$ is a basis of $T(E)$. We use $H^{F}$ for $F$ and define $A_{ij}=H^{E}_{1i}\otimes H^{F}_{1j}$ for all $i\ge 2$, $\j\ge 2$.

Now consider the action of $A_{ij}$ on variable $x_{pq}$.
\[
C_{p,q}^{i,j}:=A_{ij}(x_{pq})=(\delta_{p}^{1}- \delta_{p}^{i})(\delta_{q}^{1}-\delta_{q}^{j})
\]
\begin{eqnarray}C_{p,q}^{i,j}=
\begin{cases}
&1, \;p=q=1\\
&-1, \;p=i, q=1\\
&-1, \;p=1, q=j\\
&1, \;p=i, q=j\\
&0, \;otherwise
\end{cases}
\end{eqnarray}
Equation (\ref{E4}) implies that the matrices $C=(c_{ij})$ that stabilize $P_{\pi}$ is contained in the torus of $GL(E\otimes F)$, hence the Lie algebra of the set of such matrices is contained in the torus of $\mathfrak{gl}(E\otimes F)$, that is, it is contained in $t:=$ $Id_{E}$ $\otimes$ $Id_{F}$ $\oplus$ $T(E)$ $\otimes$ $T(F)$ $\oplus$ $T(E)$ $\otimes$ $Id_{F}$ $\oplus$ $Id_{E}$ $\otimes$ $T(F)$. Now let $L$ be an element of $t$, then $L$ can be expressed as the linear combination of $A_{ij}$'s and $Id_{E}\otimes Id_{F}$. Hence:
\[
L=(aId_{E}\otimes Id_{F}+\sum_{i,j>1}a_{ij}A_{ij} )
\]
for some $a$, $a_{ij}$ $\in$ $\mathbb{C}$.

Then
\begin{eqnarray}
L(x_{pq})=
\begin{cases}
(a+\sum_{i,j>1}a_{ij})x_{11}, &p=q=1\\
(a-\sum_{i>1}a_{iq})x_{1q}, &p=1, q\ne1\\
(a-\sum_{j>1}a_{pj})x_{p1}, &p\ne1, q=1\\
(a+a_{pq})x_{pq}, &p\ne1, q\ne1\\
\end{cases}
\end{eqnarray}
Now for a permutation $\sigma\in \mathfrak{S}_{n}$, the action of $L$ on the monomial $x_{1\sigma(1)}$ $x_{2\sigma(2)}$ $...$ $x_{n\sigma(n)}$ is:\\
if $\sigma(1)=1$,
\begin{equation}\label{EL1}
L(\prod_{p=1}^{n}x_{p\sigma(p)})
=(na+\sum_{i,j>1}a_{ij}+\sum_{p=2}^{n}a_{p\sigma(p)})\prod_{p=1}^{n}x_{p\sigma(p)}. \end{equation}
if $\sigma(1)\ne1$ and $\sigma(k)=1$,
\begin{equation}\label{EL2}
L(\prod_{p=1}^{n}x_{p\sigma(p)})
=(na+\sum_{p\ne1,k}a_{p\sigma(p)}-\sum_{i>1}a_{i\sigma(1)}-\sum_{j>1}a_{kj})\prod_{p=1}^{n}x_{p\sigma(p)}.
\end{equation}
\begin{lemma}\label{S5}
For any solution of the system of linear equations
\begin{eqnarray}
a_{ij}+a_{jk}+a_{km}=&a_{ik}+a_{kj}+a_{jm},\text{where } \{i,j,k,m\}=\{2,3,4,5\}\\
a_{ij}+a_{jk}=&a_{ij^{'}}+a_{j^{'}k}, \text{where } \{i,j,k,j^{'}\}=\{2,3,4,5\}
\end{eqnarray}
There exists a number $\lambda$ such that for any permutation $\mu$ of the set $\{2,3,4,5\}$ moving $\l$ elements,
\begin{equation}
\sum_{\scriptstyle i=2\atop \scriptstyle \mu(i)\ne i}^{5}a_{i\mu(i)}=\l\lambda
\end{equation}
\end{lemma}
\begin{proof}
check by solving this linear system.
\end{proof}
\begin{lemma}\label{lemma3}
let $n\ge 6$ be an integer, $\pi$ be a fixed partition of $n$, which is not $(1,...,1)$ or $(n)$. Assume that there exists a permutation(so a conjugacy class) $\tau$ $\in$ $\mathfrak{S}_{n}$ such that:
\begin{enumerate}
\item $\chi_{\pi}(\tau)\ne0$;
\item $\tau$ contains a cycle moving at least $4$ numbers;
\item $\tau$ fixes at least $1$ number.
\end{enumerate}
Also assume that $L(P_{\pi})=0$. Then under the above assumptions, $a=a_{ij}=0$ for all $i,j>1$.
\end{lemma}
\begin{proof}
$L(P_{\pi})=0$ means that $L$$(\prod_{p=1}^{n}x_{p\sigma(p)})$ for all $\sigma\in\mathfrak{S}_{n}$ such that $\chi_{\pi}(\sigma)\ne0$. Consider permutations $(2345...)$ $...$ $(...)$ and $(2435...)$ $...$ $(...)$(all cycles are the same except the first one, and for the first cycle, all numbers are the same except the first $4$), from formula (\ref{EL1}), we have:
\begin{equation}
na+\sum_{i,j>1}a_{ij}+a_{23}+a_{34}+a_{45}+E=0
\end{equation}
\begin{equation}
na+\sum_{i,j>1}a_{ij}+a_{24}+a_{43}+a_{35}+E=0
\end{equation}
for some linear combination $E$ of $a_{ij}$'s. Thus
\begin{equation}
a_{23}+a_{34}+a_{45}=a_{24}+a_{43}+a_{35}
\end{equation}
Similarly,
\begin{equation}
a_{ij}+a_{jk}+a_{km}=a_{ik}+a_{kj}+a_{jm}\text{, for $i,j,k,m$ distinct}.
\end{equation}
Next, consider permutations $(1234...k)...(...)$ and $(1254...k)...(...)$, again, from formula (\ref{EL2}), we obtain:
\begin{equation}
na+a_{23}+a_{34}+E'-\sum_{i>1}a_{i,2}-\sum_{j>1}a_{k,j}=0
\end{equation}
\begin{equation}
na+a_{25}+a_{54}+E'-\sum_{i>1}a_{i,2}-\sum_{j>1}a_{k,j}=0
\end{equation}
Hence,
\begin{equation}
a_{25}+a_{54}=a_{23}+a_{34}
\end{equation}
and thus
\begin{equation}
a_{ij}+a_{jk}=a_{ij^{'}}+a_{j^{'}k}\text{ ,for all $i,j,j',k$ distinct}.
\end{equation}
Now for $2\le i<j<k<m\le n$, we have system of linear equations of the same  form as Lemma \ref{S5}. So we have relations:
\[
\sum_{\scriptstyle m\ge p\ge i\atop \scriptstyle \mu(p)\ne p}a_{p\mu(p)}=\l \lambda_{ijkm}
\]
where $\mu$ is a permutation of the set $\{i,j,k,m\}$, $\l$ is the number of elements moved by $\mu$, and $\lambda_{ijkm}$ is a constant number.

It is easy to see that $\lambda_{ijkm}$ is the same for different choices of the set $\{2\le i< j< k< m\le n\}$, for example, we can compare $\{i,j,k,m\}$ and $\{i,j,k,m^{'}\}$ to obtain $\lambda_{ijkm}=\lambda_{ijkm^{'}}$. From now on, we write all $\lambda_{ijkm}$'s as $\lambda$. Hence, given any permutation of the set $\{2,...,n\}$ moving $\l$ elements,
\begin{equation}\label{relation}
\sum_{\scriptstyle n\ge p\ge 2\atop \scriptstyle \mu(p)\ne p}a_{p\mu(p)}=\l\lambda
\end{equation}
Next, we find relations among the $a_{ii}$'s for $i\ge 2$. For this purpose, consider $\tau_{1}=(243...k)...(...)$ and $\tau_{2}=(253...k)...(...)$, then
\begin{eqnarray}
&na+\sum_{i,j>1}a_{ij}+a_{24}+...+a_{k2}+a_{55}+(\text{sum of $a_{ii}$'s for $i\ne5$ fixed by $\tau_{1}$})=0\\
&na+\sum_{i,j>1}a_{ij}+a_{25}+...+a_{k2}+a_{44}+(\text{sum of $a_{ii}$'s for $i\ne4$ fixed by $\tau_{1}$})=0
\end{eqnarray}
Combine these two equations and equation (\ref{relation}) to obtain
\[
a_{44}=a_{55}.
\]
The same argument implies that $a_{ii}=a_{55}$ for all $n\ge i\ge2$. Now we have:
\[
\sum_{i,j>1}a_{ij}=\sum_{1<i<j}(a_{ij}+a_{ij})+2(n-1)a_{55}=(n-1)(n-2)\lambda+2(n-1)a_{55}
\]
Let $\sigma=(1)$, formula(\ref{EL1}) implies
\begin{equation}\label{equation1}
na+(n-1)(n-2)\lambda+3(n-1)a_{55}=0
\end{equation}
Let $\sigma=(2345...)...(...)$($\sigma(1)=1$), again by formula (\ref{EL1})
\begin{equation}\label{equation2}
na+((n-2)(n-1)+\l)\lambda+(3(n-1)-\l)a_{55}=0
\end{equation}
where $\l$ is the number of elements moved by $\sigma$.
Let $\sigma_{1}=(123...p4)...(...)$ and $\sigma_{2}=(143...p2)...(...)$. Then formula (\ref{EL2}) gives:

\begin{eqnarray}\label{E5}
0=&na+(a_{23}+...+a_{p4})+\tilde{E} -\sum_{i>1}a_{i2}-\sum_{j>1}a_{4j}\\
\label{E6} 0=&na+(a_{43}+...+a_{p2})+\tilde{E}-\sum_{i>1}a_{i4}-\sum_{j>1}a_{2j}
\end{eqnarray}
Note that $\tilde{E}$ comes from the product of disjoint cycles in $\sigma_{1}$ and $\sigma_{2}$ except the first one, so they are indeed the same, and if we assume that $\sigma_{1}$ moves $\l^{'}$ elements, and the first cycle in $\sigma_{1}$ moves $r$ elements, then $\tilde{E}$ $=(\l^{'}-r)\lambda+(n-\l^{'})a_{55}$. On the other hand,
\[
\begin{aligned}
a_{23}+...+a_{p4}&=a_{23}+...+a_{p4}+a_{42}-a_{42}=r\lambda-a_{42}\\
a_{43}+...+a_{p2}&=a_{43}+...+a_{p2}+a_{24}-a_{24}=r\lambda-a_{24}
\end{aligned}
\]
Equations (\ref{E5}) and (\ref{E6}) gives:
\begin{equation}\label{equation3}
na+(\l^{'}+2n-5)\lambda+(n-\l^{'}+2)a_{55}=0
\end{equation}
Now equations (\ref{equation1}), (\ref{equation2}), and (\ref{equation3}) imply that $\lambda=a=a_{55}=0$. From equation (\ref{E5}), we have:
\[
a_{42}=\sum_{j>1}a_{4j}+\sum_{i>1}a_{i2}
\]
similarly,
\[
a_{k2}=\sum_{j>1}a_{kj}+\sum_{i>1}a_{i2}\text{ for all } n\ge k\ge 2
\]
The sum of these equations is:
\[
\sum_{i>1}a_{i2}=(n-1)\sum_{i>1}a_{i2}+\sum_{i,j>1}a_{ij}=(n-1)\sum_{i>1}a_{i2}.
\]
Hence $\sum_{i>1}a_{i2}=0$. For the same reason $\sum_{j>1}a_{4j}=0$, therefore $a_{42}=0$. By the same argument, $a_{ij}=0$ for all $n\ge i\ne j \ge 2$, and this completes the proof of the lemma.
\end{proof}
\begin{lemma}\label{lemma4}
If $n\ge 6$, then for any partition $\lambda$ of $n$, except $\lambda=(3,1,1,1)$, $\lambda=(4,1,1)$ and $\lambda=(4,1,1,1)$, there exists a permutation $\tau\in\mathfrak{S}_{n}$ satisfying conditions (1), (2), (3) in Lemma \ref{lemma3}.
\end{lemma}
\begin{proof}
Write $\lambda=(\lambda_{1},\lambda_{2},...,\lambda_{p})$ where $\lambda_{1}\ge\lambda_{2}\ge...\lambda_{p}\ge 1$ and $\sum_{i=1}^{p}\lambda_{i}=n$. Without loose of generality, we may assume $p\ge \lambda_{1}$, otherwise, we can consider the conjugate $\lambda^{'}$ of $\lambda$. There exists a largest integer $m$ such that the Young diagram of $\lambda$ contains an $m\times m$ square.

Now we will construct $\tau$ using the Murnagham-Nakayama Rule (See \cite{MR1824028})case by case:
\begin{enumerate}
\item If $m=1$ then $\lambda$ is a hook:$(\lambda_{1},1,...,1)$, there are the following cases:
\begin{enumerate}
\item $p>\lambda_{1}$ and $\lambda_{1}\ge 4$. Take $\tau=(p-1,1^{n-p+1})$ then $\chi_{\lambda}(\tau)\ne 0$ by the Murnagham-Nakayama Rule. In this case, $n\ge 8$.
\item $p>\lambda_{1}$ and $\lambda_{1}=1$. This case is trivial.
\item $p>\lambda_{1}$ and $\lambda_{1}=2$ or $3$. $\tau=(4,1^{n-4})$ will work.
\item $p=\lambda_{1}$. Take $\tau=(p-1,1^{n-p+1})$ if $p\ge 6$ is even and $\tau=(p-2,1^{n-p+2})$ if $p\ge 7$ is odd. In this case $n\ge 11$.
\item $p=\lambda_{1}=5$. $\tau$ exists by checking the character tables.
\end{enumerate}
\item If $m\ge2$, let $\xi$ be the length of the longest skew hook contained in the young diagram of $\lambda$. Then take $\tau=(\xi,1^{n-\xi})$.
\end{enumerate} \end{proof}
\begin{proof}[\textbf{proof of Theorem \ref{theorem1}}]
For the case $n=5$, one can check directly. By Lemma \ref{lemma3} and lemma \ref{lemma4}, we know that for $n\ge 6$ and $\pi$ not equal to $(3,1,1,1)$ and $(4,1,1,1)$, the lie algebra of $\{C\in M_{n\times n}\mid P_{\pi}(C*X)=P_{\pi}(X)\}$ is $T(E)\otimes Id_{F}\oplus Id_{E}\otimes T(F)$, so the identity component of $\{C\in M_{n\times n}\mid P_{\pi}(C*X)=P_{\pi}(X)\}$ is $T(GL(E)\times GL(F))$, and hence the identity component of $G(P_{\pi})$ is $\Delta(\mathfrak{S}_{n})$ $\ltimes T(GL(E)\times GL(F))$ $\ltimes \mathbb{Z}_{2}$. For cases $\pi=(3,1,1,1)$, $\pi=(4,1,1)$ the statement is true by Theorem \ref{theorem2}.
\end{proof}
By investigating the equation ($\ref{E4}$), we can give a sufficient condition for the stabilizer of $P_{\pi}$ to be $\Delta(\mathfrak{S}_{n})$ $\ltimes T(GL(E)\times GL(F))$ $\ltimes \mathbb{Z}_{2}$ as follows:
\begin{lemma}\label{prop}
Let $\pi$ be a partition of $n$ which is not $(1,...,1)$ or $(n)$. Assume that there exist permutations $\sigma$, $\tau$ $\in$¡¡$\mathfrak{S}_{n}$ and an integer $p\ge2$, such that $\chi_{\pi}((i_{1}...i_{p})\sigma)$ $\ne$ $0$, $\chi_{\pi}((1i_{1}...i_{p})\sigma)$ $\ne$ $0$, $\chi_{\pi}(\tau)$ $\ne$ $0$ and $\chi_{\pi}((ij)\tau)$ $\ne$ $0$, where $(i_{1}...i_{p})$ and $(1i_{1}...i_{p})$ are cycles disjoint from $\sigma$, and $(ij)$ is disjoint from $\tau$. Then the stabilizer of $P_{\pi}$ is $\Delta(\mathfrak{S}_{n})$ $\ltimes T(GL(E)\times GL(F))$ $\ltimes \mathbb{Z}_{2}$.
\end{lemma}
\begin{proof}
For convenience, we will show for the case $\sigma=(1)$ and $p=2$, the other cases are similar. In equations \ref{E4}, let $\sigma=(ij)$ and $(1ij)$, where $1< i,j\leqslant n$ and $i\ne j$. Then
\begin{equation}
c_{ji}c_{ij}c_{11}\prod_{k\ne 1,i,j}c_{kk}=1
\end{equation}
\begin{equation}
c_{ji}c_{1j}c_{i1}\prod_{k\ne 1,i,j}c_{kk}=1
\end{equation}
So $c_{ij}=\frac{c_{i1}c_{1j}}{c_{11}}$. Similarly, the existence of $\tau$ will give the relation $c_{ii}=\frac{c_{i1}c_{1i}}{c_{11}}$ for all $1\le i \le n $. Set
\[\mathbf{A}=
\begin{pmatrix}a_{1}&0&\dots&0\\
      0&a_{2}&\dots&0\\
      \vdots&\vdots&\ddots&\vdots\\
      0&0&\dots&a_{n}
      \end{pmatrix},
\mathbf{B}=\begin{pmatrix}b_{1}&0&\dots&0\\
      0&b_{2}&\dots&0\\
      \vdots&\vdots&\ddots&\vdots\\
      0&0&\dots&b_{n}
      \end{pmatrix}
\]\\
Where $a_{ii}=c_{i1}$ and $b_{jj}=\frac{c_{1j}}{c_{11}}$. So we have \[C\ast X=AXB, \text{ with } det(AB)=1.\]
\end{proof}

The following two propositions guarantee the existence of permutations satisfied conditions in Lemma (\ref{prop}).
\begin{proposition}\label{3.1}
Let $n\ge 3$, and $\pi$ be a non-symmetric partition of $n$, then there exists nonnegative integers $k_{1}$,..., $k_{r}$ such that $k_{1}$ $+...+$ $k_{r}$ $=$ $n-2$, such that $|(\chi_{\pi}(\tau))|=1$, where $\tau$ is a permutation of type $(k_{1}$,..., $k_{r},1^{2})$ or $(k_{1}$,..., $k_{r},2)$.
\end{proposition}
\begin{proof}
See Proposition ($3.1$), Coelho, M. Purifica{\c{c}}{\~a}o and Duffner, M. Ant{\'o}nia \cite{MR1674232}.
\end{proof}
\begin{proposition}\label{3.2}
Let $n>4$ and $\pi$ be a non-symmetric partition of $n$, then there exists nonnegative integers $k_{1}$ $,...,$ $k_{r}$ and $q$ with $k_{r}>1$, $q\ge1$, and $k_{1}$ $+...+$ $k_{r}$ $+q$ $=n$, such that $\chi_{\pi}(\sigma)\ne 0$, where $\sigma$ $\in$ $\mathfrak{S}$ is of type $(k_{1} ,..., k_{r},1^{q})$ or $(k_{1} ,..., k_{r}+1,1^{q-1})$.
\end{proposition}
\begin{proof}
See Proposition ($3.2$), Coelho, M. Purifica{\c{c}}{\~a}o and Duffner, M. Ant{\'o}nia \cite{MR1674232}.
\end{proof}
\begin{proof}[\textbf{proof of Theorem \ref{theorem2}}]
Since $n\ge 5$, by propositions \ref{3.1} and \ref{3.2}, there exist permutations satisfying conditions in Lemma \ref{prop}, then the theorem follows.
\end{proof}

\leftline{\textbf{Acknowledgement}}

\leftline{I would like to thank the anonymous referee for his/her valuable comments.}

\bibliographystyle{amsplain}

\bibliography{mybib}
\end{document}